\documentclass[reqno]{amsart}
\usepackage{amsmath}
\usepackage{amsthm}
\usepackage{amsfonts}
\usepackage{amssymb}
\usepackage{graphicx}
\usepackage{mathrsfs}
\usepackage{cite}
\usepackage{textcomp}

\title[STOLARSKY'S INVARIANCE PRINCIPLE FOR PROJECTIVE SPACES]{STOLARSKY'S INVARIANCE PRINCIPLE FOR PROJECTIVE SPACES, II}

\author[M.M. SKRIGANOV]{M.M. SKRIGANOV}

\address{St. Petersburg Department of the Steklov Mathematical Institute 
of the Russian Academy of Sciences, 
27, Fontanka, St.Petersburg 191023, Russia}

\email{mmskrig@gmail.com}



\numberwithin{equation}{section}

\newtheorem{theorem}{Theorem}[section]
\newtheorem{lemma}{Lemma}[section]
\newtheorem{proposition}{Proposition}[section]
\newtheorem{corollary}{Corollary}[section]
\theoremstyle{remark}

\theoremstyle{remark}
\newtheorem{definition}{Definition}[section]

\def\dd{\mathrm{d}}

\def\Cc{\mathbb{C}}

\def\Ff{\mathbb{F}}

\def\Hh{\mathbb{H}}

\def\Oo{\mathbb{O}}

\def\Rr{\mathbb{R}}

\def\BBB{\mathcal{B}}
\def\CCC{\mathcal{C}}
\def\DDD{\mathcal{D}}
\def\EEE{\mathcal{E}}
\def\FFF{\mathcal{F}}

\def\III{\mathcal{I}}

\def\MMM{\mathcal{M}}

\DeclareMathOperator{\diam}{diam}

\renewcommand{\le}{\leqslant}
\renewcommand{\ge}{\geqslant}
\renewcommand{\Re}{\mathop{\mathrm{Re}}\nolimits}
\renewcommand{\Im}{\mathop{\mathrm{Im}}\nolimits}

\numberwithin{equation}{section}

\DeclareMathOperator{\Spin}{\mathrm{Spin}}

\def\M{\mathcal M}

\def\la{\lambda}

\def\E{\mathcal E}
\def\O{\mathcal O}

\numberwithin{equation}{section}
\theoremstyle{plain}

\newcommand{\bp}{\begin{proof}}
\newcommand{\ep}{\end{proof}}
\newcommand{\bl}{\begin{lemma}}
\newcommand{\el}{\end{lemma}}
\newcommand{\bt}{\begin{theorem}}
\newcommand{\et}{\end{theorem}}
\newcommand{\bd}{\begin{definition}}
\newcommand{\ed}{\end{definition}}
\newcommand{\ba}{\begin{arrow}}
\newcommand{\ea}{\end{arrow}}

\begin{document}


%
%
%
%




\begin{abstract}
It was proved in the first part of this work \cite{0} that Stolarsky's invariance
principle, known previously for point distributions on the Euclidean 
spheres \cite{33}, 
can be extended to the real, complex, and quaternionic  
projective spaces and the octonionic projective plane. 
The geometric features of these spaces have been used very essentially 
in the proof. In the present paper, relying on the theory of spherical 
functions on such spaces, we give an alternative analytic
proof of Stolarsky's invariance principle. 
\end{abstract}



\maketitle

\thispagestyle{empty}



\section{Introduction}\label{sec1}

In 1973 Kenneth~B.~Stolarsky \cite{33} established the following remarkable formula for point distributions on the Euclidean spheres.
Let $S^d=\{x\in \Rr^{d+1}:\Vert x\Vert=1\}$ be the standard $d$-dimensional unit sphere in $\Rr^{d+1}$
with the geodesic (great circle) metric $\theta$ and the Lebesgue measure
$\mu$ normalized by $\mu(S^d)=1$. 
We write $\CCC(y,t)=\{x\in S^d: (x,y) > t\}$ for the spherical cap
of height $t\in [-1,1]$ centered at 
$y\in S^d$. Here we write $(\cdot,\cdot)$ and $\Vert \cdot \Vert$ for 
the inner product and the Euclidean norm in $\Rr^{d+1}$.

For an $N$-point subset $\DDD_N\subset S^d$,
the spherical cap quadratic discrepancy is defined by
\begin{equation}
\lambda^{cap}[\DDD_N]=
\int_{-1}^1\int_{S^d}\left(\,\# \{|\CCC(y,t)\cap \DDD_N\}-N\mu(\CCC(y,t))\,\right)^2\dd\mu(y)\,
\dd t.
\label{eq1.1*}
\end{equation}

We introduce the sum of pairwise Euclidean distances 
between points of $\DDD_N$
\begin{equation}
\tau [\DDD_N]=\frac 12 \sum\nolimits_{x_1, x_2\in \DDD_N} \Vert x_1-x_2 \Vert
= \sum\nolimits_{x_1, x_2\in \DDD_N} \sin\frac 12 \theta (x_1,x_2),
\label{eq1.2*}
\end{equation}
and write $\langle \tau \rangle$ for the average value of the Euclidean
distance on $S^d$,       
\begin{equation}
\langle \tau \rangle =\frac 12 \iint\nolimits_{S^d\times S^d} \Vert y_1-y_2\Vert \,
d\mu (y_1) \, \dd\mu (y_2).
\label{eq1.12*}
\end{equation}

The study of the quantities \eqref{eq1.1*} and \eqref{eq1.2*} falls within 
the subjects of discrepancy theory and geometry of distances,
see  \cite{2, 6} and references therein.  
It turns out that the quantities \eqref{eq1.1*} and \eqref{eq1.2*} are 
not independent and are intimately 
related by the following remarkable identity
\begin{equation}
\gamma(S^d)\lambda^{cap}[\DDD_N]+\tau [\DDD_N]=\langle \tau \rangle N^2,
\label{eq1.3**}
\end{equation}
for an arbitrary $N$-point subset $\DDD_N\subset S^d$. Here $\gamma(S^d)$
is a positive constant independent of $\DDD_N$,
\begin{equation}
\gamma(S^d)=\frac{d\,\sqrt{\pi}\,\,\Gamma(d/2)}{2\,\Gamma((d+1)/2)}\, .
\label{eq1.3***}
\end{equation}

The identity \eqref{eq1.3**} is
known in the literature as \textit{Stolarsky's invariance principle}. 
The original proof of \eqref{eq1.3**} has been essentially simplified 
in \cite{8, 11}.
Its applications to the discrepancy theory and geometry of distances on $S^d$
have been given in \cite{33} and \cite{5, 6}.


The spheres $S^d$ are the simplest examples of compact
Riemannian symmetric manifolds of rank one. 
In the first part of this work \cite{0}  Stolarsky's invariance
principle \eqref{eq1.3**} was proved for all such spaces.
The proof was based on geometric features of these spaces.
The aim of the present paper is to give an alternative pure analytic
proof relying on the theory of spherical functions.

The paper is organized as follows. 
To make the present paper essentially self-contained, we shall recall
in Section~2 some general facts from \cite{0} on relationships between
discrepancies and metrics. The main results are stated in Section~3.
In Section~4, using spherical functions, we reduce the proof of main results
to  calculation of certain integrals with Jacobi polynomials.
These integrals are calculated explicitly in Section~5 with the help of
Watson's theorem for the generalized hypergeometric series $_3F_2(1)$.

\section{$L_1$-invariance principles}\label{sec2}

Let $\M$ be a compact metric measure space with a fixed metric $\theta$ 
and a finite Borel measure $\mu$, normalized, for convenience, by 
\begin{equation}
\diam (\MMM,\theta)=\pi, \quad \mu (\MMM)=1,
\label{eq1.1}
\end{equation}
where 
$\diam (\EEE,\rho)=\sup \{ \rho(x_1,x_2): x_1,x_2\in \EEE\}$
denotes the diameter of a subset $\EEE\subseteq \MMM$  with respect to 
a metric $\rho$. 

We write $\BBB(y,r)=\{x\in\MMM:\theta (x,y)<r\}$ for the ball of radius $r\in \III$
centered at  $y\in \MMM$ and of volume $v(y,r)=\mu (\BBB(y,r))$. 
Here $\III = \{r=\theta(x_1,x_2): x_1,x_2\in \MMM\}$ denotes the set of all possible radii. 
If the space $\MMM$ is connected, we have $\III = [\,0,\pi\,]$.

We consider distance-invariant metric spaces. Recall that a metric space $\MMM$ is called distance-invariant, if the volume  of any
ball $v(r)=v(y,r)$ is independent of $y\in \MMM$.
The typical examples of distance-invariant spaces are homogeneous spaces $\MMM=G/K$, 
where $G$ is a compact group, $K\subset G$ is a closed subgroup, and a metric
$\theta$ and a measure $\mu$ on $\MMM$ are $G$-invariant. 

For an $N$-point subset $\DDD_N\subset \MMM$, the ball quadratic discrepancy is 
defined by
\begin{equation}
\lambda[\xi,\DDD_N]=
\int_{\III}\int_{\MMM}\left(\,\# \{\BBB(y,r)\cap \DDD_N\}-Nv(r))\,\right)^2\,
\dd\mu(y)\, \dd \xi (r),
\label{eq1.3*}
\end{equation}
where $\xi$ is a finite measure on the set of radii $\III$.

Notice that for $S^d$ spherical caps and balls are 
related by  $\CCC(y,t)=\BBB(y,r)$, $t=\cos r$, and  the 
discrepancies \eqref{eq1.1*} and \eqref{eq1.3*} are related by
$\lambda^{cap}[\DDD_N]=\lambda[\xi^{\natural},\DDD_N]$, where  
$\dd\xi^{\natural}(r)=\sin r\,\dd r,\, r\in \III =[0,\pi]$.

The ball quadratic discrepancy \eqref{eq1.3*} can be written in the form
\begin{equation}
\lambda [\xi, \DDD_N] =
\sum\nolimits_{x_1,x_2\in \DDD_N} \la (\xi, x_1,x_2)
\label{eq1.7}
\end{equation}
with the kernel 
\begin{equation}
\lambda(\xi,x_1,x_2)=\int_\III\int_\MMM
\Lambda  (\BBB(y,r),x_1)\,\Lambda (\BBB(y,r),x_2)
\, \dd\mu (y)\,\dd\xi (r)\, ,
\label{eq1.6}
\end{equation}
where 
\begin{equation}
\Lambda  (\BBB(y,r),x)=\chi(\BBB(y,r),x)- v(r),
\label{eq1.6**}
\end{equation}
and $\chi(\E,\cdot)$ denotes the characteristic function of a subset $\E\subseteq\M$.
Notice that 
the symmetry of metric $\theta$ implies the relation   
$\chi (\BBB(y,r),x)=\chi (\BBB(x,r),y)$.

Substituting \eqref{eq1.6**} into \eqref{eq1.6}, we obtain
\begin{equation}
\lambda(\xi,x_1,x_2)=\int_\III
\Big (\mu  (\BBB(x_1,r)\cap \BBB(x_2,r)) -v(r)^2 \Big)
\, \dd\xi (r),
\label{eq1.6*}
\end{equation}
since the space $\MMM$ is distance-invariant. 

For an arbitrary metric $\rho$ on $\MMM$ we introduce 
the sum of pairwise distances  
\begin{equation}
\rho [\DDD_N] =\sum\nolimits_{x_1,x_2\in \DDD_N} \rho (x_1,x_2).
\label{eq1.10}
\end{equation}
and the average value 
\begin{equation}
\langle \rho \rangle =\int\nolimits_{\MMM\times\MMM} \rho (y_1,y_2) \,
\dd\mu (y_1) \, \dd\mu (y_2).
\label{eq1.12}
\end{equation}
We introduce the following symmetric difference metrics on the space $\M$
\begin{align}
\theta^{\Delta} (\xi, y_1,y_2) & =\frac 12\int_\III
\mu (\BBB(y_1,r)\Delta \BBB(y_2,r))
\, \dd\xi(r) \notag
\\
& =\frac 12\int_\III\int_{\MMM}\chi(\BBB(y_1,r)\Delta \BBB(y_2,r),y)
\,\dd\mu(y)\,\dd\xi(r),
\label{eq1.13}
\end{align}
where
$\BBB(y_1,r)\Delta \BBB(y_2,r)=\BBB(y_1,r)\cup \BBB(y_2,r) \setminus 
\BBB(y_1,r)\cap \BBB(y_2,r)$
is the symmetric difference of the balls  $\BBB(y_1,r)$ and $\BBB(y_2,r)$.
We have
\begin{align*}
\chi (\BBB(y_1,r)\Delta \BBB(y_2,r))   
=\chi(\BBB(y_1,r),y) +\chi 
(\BBB(y_2,r),y) -2\chi (\BBB(y_1,r),y) \chi (\BBB(y_2,r),y).
\end{align*}
Substituting this equality into  \eqref{eq1.13}, 
we obtain 
\begin{equation}
\theta^{\Delta}(\xi,y_1,y_2) 
=\int_{\III} \Big(v(r)-\mu (\BBB(y_1,r)\cap \BBB(y_2,r))\Big)\, \dd \xi (r),
\label{eq1.14*}
\end{equation}
and 
\begin{align}
\langle \theta^{\Delta}(\xi) \rangle  
= \int_{\III}\Big(v(r)-v(r)^2\Big)\, \dd\xi (r).
\label{eq1.19}
\end{align}

On the other hand, we also have 
\begin{equation}
\chi(\BBB(y_1,r)\Delta \BBB(y_2,r),y) = 
\vert\chi (\BBB(y_1,r),y)-\chi(\BBB(y_2,r),y)\vert.
\label{eq1.15}
\end{equation}
Therefore,
\begin{equation}
\theta^{\Delta} (\xi, y_1,y_2)=
\frac 12\int_\III\int_{\MMM}\vert\chi (\BBB(y_1,r),y)-\chi(\BBB(y_2,r),y)\vert
\,\dd\mu(y)\,\dd\xi(r)
\label{eq1.15*}
\end{equation}
is an $L_1$-metric.

 In line with the definition \eqref{eq1.10}, we put 
\begin{equation*}
\theta^{\Delta} [\xi,\DDD_N] =\sum\nolimits_{x_1,x_2\in \DDD_N} 
\theta^{\Delta} (\xi,x_1,x_2).
\end{equation*}
Comparing the relations \eqref{eq1.6*}, \eqref{eq1.14*} and \eqref{eq1.19},
we arrive at the following.
\begin{proposition}\label{prop1.1}
Let a compact metric measure space $\M$ 
be distance-invariant, then we have
\begin{align}
\lambda (\xi,y_1,y_2)+\theta^{\Delta}(\xi ,y_1,y_2) 
= \langle 
\theta^{\Delta} (\xi)\rangle .
\label{eq1.30}
\end{align}
In particular, we have the following $L_1$-invariance principle
\begin{equation}
\la [\,\xi,\DDD_N\,]+\theta^{\Delta}[\,\xi , \DDD_N\,]  = \langle 
\theta^{\Delta} (\xi)\rangle \, N^2
\label{eq1.31}
\end{equation}
for an arbitrary $N$-point subset $\DDD_N\subset \MMM$.  
\end{proposition}

In the case of spheres $S^d$,  relations of the type \eqref{eq1.30} 
and \eqref{eq1.31} were given in \cite{33}. Their extensions to more general
metric measure spaces were given in 
\cite[Theorem~2.1]{30},
\cite[Eq.~(1.30)]{31} and \cite[Proposition~1.1]{0}.

Recall that a metric space $\MMM$ with a metric $\rho$ is called isometrically
$L_q$-embeddable ($q=1\,\, \mbox {or}\,\, 2$), if there exists 
a map $\varphi:\MMM\ni x\to \varphi(x)\in L_q$, such that
$\rho(x_1,x_2)=\Vert\varphi(x_1)-\varphi(x_2)\Vert_{L_q}$ for all $x_1$, 
$x_2\in \M$. Notice that the  $L_2$-embeddability 
is stronger and
implies the $L_1$-embeddability, see~\cite[Sec.~6.3]{17}. 

Since the space $\MMM$ is isometrically $L_1$-embeddable with respect to
the symmetric difference metrics $\theta^{\Delta} (\xi)$, see \eqref{eq1.15*},
the identity \eqref{eq1.31} is called the $L_1$-invariance principle. 
At the same time, Stolarsky's invariance principle should be called 
the $L_2$-invariance principle, because it involves the Euclidean metric.

\section{$L_2$-invariance principles}\label{sec3}

It is proved in \cite{0} that the $L_2$-invariance principles holds for compact
Riemannian symmetric manifolds of rank one. All these manifolds are
completely classified, see, for example,~\cite[Chap.3]{7}.
They are homogeneous spaces $\MMM=G/K$, where  $G$ and $K\subset G$ are
compact Lie groups. The complete list of all compact
Riemannian symmetric manifolds of rank one is the following: 

(i) The $d$-dimensional Euclidean spheres 
$S^d=SO(d+1)/SO(d)\times 
\{1\}$, $d\ge 2$, and $S^1=O(2)/O(1) \times \{ 1\}$. 

(ii) The real projective spaces $\Rr P^n=O(n+1)/O(n)\times O(1)$.

(iii) The complex projective spaces $\Cc P^n=U(n+1)/U(n)\times U(1)$.

(iv) The quaternionic projective spaces $\Hh P^n=Sp(n+1)/Sp(n)\times Sp(1)$,

(v) The octonionic projective plane $\Oo P^2=F_4/\Spin (9)$.

Here we use the standard notation from the theory of Lie groups; in particular,  
$F_4$ is one of the exceptional Lie groups in Cartan's classification. 

The indicated projective spaces $ \Ff P^n $ as compact Riemannian manifolds have dimensions $d$, 
\begin{equation}
d=\dim_{\Rr} \Ff P^n=nd_0, \quad d_0=\dim_{\Rr}\Ff,
\label{eq2.1} 
\end{equation}
where $d_0=1,2,4,8$ for $\Ff=\Rr$, $\Cc$, $\Hh$, $\Oo$, correspondingly.

For the spheres $S^d$ we put $d_0=d$ by definition. Projective spaces 
of dimension  $d_0$ ($n=1$) are homeomorphic to the spheres $S^{d_0}$:
$\Rr P^1 \approx S^1, \Cc P^1 \approx S^2,  \Hh P^1 \approx S^4, 
\Oo P^1 \approx S^8$.
We can conveniently agree that $d>d_0$ ($n\ge 2)$ for projective spaces,
while the equality $d=d_0$ holds only for spheres. Under this convention,
the dimensions $d=nd_0$ and $d_0$ define uniquely (up to homeomorphisms) 
the corresponding homogeneous space which we denote by $Q=Q(d,d_0)$.

We consider $Q(d,d_0)$ as a metric measure space with the metric $\theta$
and measure $\mu$ proportional to the invariant Riemannian distance and measure
on $Q(d,d_0)$. The coefficients of proportionality are defined to satisfy \eqref{eq1.1}.
In what follows we always assume that $n=2$ if $\Ff=\Oo$, since  
projective spaces $\Oo P^n$ do not exist for $n>2$. 

Any space $Q(d,d_0)$ is distance-invariant and the volume of balls is given by 
\begin{equation}
v(r)=B(d/2,d_0/2)^{-1}
\int^r_0(\sin\frac{1}{2}u)^{d-1}(\cos \frac{1}{2}u)^{d_0-1}\,\dd u, 
\quad r\in [\,0,\pi\,],
\label{eq2.2}
\end{equation}
where
$\kappa(d,d_0)=B(d/2,d_0/2)^{-1}$;
$B(a,b)=\Gamma(a)\Gamma(b)/\Gamma(a+b)$ and 
$\Gamma(a)$ are the beta and gamma functions.
Equivalent forms of  \eqref{eq2.2}
can be found in the literature, see, for example, \cite[pp.~177--178]{19}.

The chordal metric on the spaces $Q(d,d_0)$ can be defined by
\begin{equation}
\tau(x_1,x_2)=\sin \frac{1}{2}\theta(x_1,x_2)
= \sqrt{\frac{1-\cos (x_1, x_2)}{2}}\, , \quad x_1,x_2\in Q(d,d_0).
\label{eq2.4}
\end{equation}
Notice that the expression \eqref{eq2.4} defines a metric because the function 
$\varphi(\theta)=\sin \theta/2$, $0\le \theta\le \pi$, 
is concave, increasing, and  $\varphi(0)=0$, that implies the triangle 
inequality.
For the sphere $S^d$ we have
$\cos \theta(x_1,x_2)=(x_1,x_2),\, x_1,x_2\in S^d$ and
\begin{equation}
\tau(x_1,x_2)=\sin \frac12\theta(x_1,x_2)=\frac{1}{2}\,\Vert x_1-x_2\Vert\, .
\label{eq2.6*}
\end{equation}
Each projective space $\Ff P^n$ can be canonically embedded into the unit sphere
\begin{equation}
\Pi:Q(d,d_0)\ni x\to \Pi(x)\in S^{m-1}\subset \Rr^m, \quad 
m=\frac{1}{2}(n+1)(d+2),
\label{eq2.6}
\end{equation}
such that
\begin{equation}
\tau(x_1,x_2)=\frac{1}{\sqrt{2}}\|\Pi(x_1)-\Pi(x_2)\|, \quad x_1,x_2\in 
\Ff P^n,
\label{eq2.7}
\end{equation}
where $\|\cdot \|$ is the Euclidean norm in $\Rr^{m}$.
Hence, the metric $\tau(x_1,x_2)$ is proportional to the Euclidean length of a segment
joining the corresponding points $\Pi(x_1)$ and $\Pi(x_2)$ on the unit sphere 
and normalized by $\diam (Q(d,d_0),\tau)=1$.
The chordal metric $\tau$ on the complex projective space $\Cc P^n$ 
is known as the Fubini--Study metric.
The embedding \eqref{eq2.6} can be described explicitly in terms of
idempotents of division algebras (of real and complex numbers, quaternions and 
octonions), see \cite{0}. In the present paper such a 
detail description
of the spaces $Q(d, d_0)$ will not be needed.
 

\begin{theorem}\label{thm3.1}
	For any space $Q=Q(d,d_0)$, we have the equality 
	\begin{equation}
	\tau(x_1,x_2)=\gamma(Q)\,\, \theta^{\Delta}(\xi^{\natural} ,x_1,x_2). 
	\label{eq2.8}
	\end{equation}
	where $\dd\xi^{\natural}(r)=\sin r\,\dd r$, $r\in [0,\pi]$ and
	\begin{equation}
	\gamma(Q)
	=\frac{\sqrt{\pi}}{4}\,(d+d_0)\,
	\frac{\Gamma(d_0/2)}{\Gamma((d_0+1)/2)}
	=\frac{d+d_0}{2d_0}\,\gamma(S^{d_0})\, ,
	\label{eq1.33*}
	\end{equation}
	where $\gamma(S^{d_0})$ is defined by \eqref{eq1.3***}. 
\end{theorem}

Theorem~3.1 summarizes Theorems~1.1 and~1.2 from \cite{0}, its new proof is 
given below in Section~4.

Notice that the constant $\gamma(Q)$ has the following geometric 
interpretation
\begin{equation}
\gamma(Q)=\frac{\langle \tau\rangle}{\langle 
	\theta^{\Delta}(\xi^{\natural})\rangle}=\frac{\diam 
	(Q,\tau)}{\diam(Q,\,\theta^{\Delta}(\xi^{\natural}))}\, .
\label{eq2.9}
\end{equation}
Indeed, it suffices to calculate the average values \eqref{eq1.12} of both 
metrics 
in \eqref{eq2.8} to obtain 
the first equality in \eqref{eq2.9}. Similarly, writing \eqref{eq2.8} for any 
pair of antipodal points $x_1$, $x_2$, $\theta(x_1,x_2) = \pi$, we obtain the 
second 
equality in \eqref{eq2.9}. 

Comparing the Theorem~3.1 and Proposition~2.1, we arrive at the following,
see \cite[Corollary~1.1]{0}.
\begin{corollary}\label{cor3.1}
For any space $Q=Q(d,d_0)$, we have 
\begin{equation}
\gamma(Q)\,\lambda (\xi^{\natural},x_1, x_2)+\tau (x_1, x_2)=\langle \tau\rangle, 
\label{eq2.10aa}
\end{equation}
In particular, we have the following $L_2$-invariance principle
\begin{equation}
\gamma(Q)\,\lambda [\xi^{\natural},\DDD_N]+\tau [\DDD_N]=\langle \tau\rangle 
N^2
\label{eq2.10}
\end{equation}
for an arbitrary $N$-point subset $\DDD_N\subset Q$.
\end{corollary}

In the case of spheres $S^d$,
the identity \eqref{eq2.10} coincides with \eqref{eq1.3**}. 
The identity \eqref{eq2.10} can be thought of as an extension of
Stolarsky's invariance principle to all compact
Riemannian symmetric manifolds of rank one.



The average value $\langle \tau \rangle$ of the chordal metric $\tau$ can be easily calculated with the help of the formulas \eqref{eq1.12} and \eqref{eq2.2}:
\begin{equation}
\langle \tau\rangle 
= B(d/2, d_0 /2)^{-1}\, B((d+1)/2, d_0 /2)\, .
\label{eq0aaa}
\end{equation}

Applications of the invariance principle \eqref{eq2.10} to the discrepancy 
theory 
and distance geometry on projective spaces have been given and discussed  in 
\cite{31, 0}.
 It is worth noting that the equality \eqref{eq2.8} is of interest by itself. 
Since the integrand in \eqref{eq1.15*} takes the values 0 and 1 only, we can write
\begin{equation}
\theta^{\Delta} (\xi, y_1,y_2)=\Big (\theta_p^{\Delta} (\xi, y_1,y_2)\Big )^p,
\,\, p>0,
\label{eq1.15*a}
\end{equation}
where
\begin{equation}
\theta_p^{\Delta} (\xi, y_1,y_2)=
\Big (\frac 12\int^{\pi}_0\int_{Q}
\vert\chi (\BBB(y_1,r),y)-\chi(\BBB(y_2,r),y)\vert^p )
\,\dd\mu(y)\,\dd\xi(r) \Big )^{1/p},
\label{eq1.15*aa}
\end{equation}
is an $L_p$-metric for $p\ge 1$.

Comparing \eqref{eq2.8} and
\eqref{eq1.15*a}, we see that the chordal metric $\tau$ is proportional
to the $p$-th power of the metric $\theta_p^{\Delta}(\xi^{\natural})$ 
for all $p\ge 1$. This is a nontrivial fact. For $p=2$, we have 
\begin{equation}
\tau (x_1,x_2)=
\frac {\gamma (Q)}{2}\int^{\pi}_0\int_{Q}
\vert\chi (\BBB(x_1,r),y)-\chi(\BBB(x_2,r),y)\vert^2 
\,\dd\mu(y)\,\dd\xi^{\natural}(r) ,
\label{eq1.15**aa}
\end{equation}
In particular, the equality \eqref{eq1.15**aa} implies the existence of
Gaussian random fields on the spaces $Q(d,d_0)$, see \cite{19, 31}. 




\section{Proof of Theorem~3.1}\label{sec5}

The zonal spherical functions $\phi_l (x_1, x_2),\, x_1,x_2\in Q,$ for 
the spaces $Q=Q(d,d_0)$ 
form an orthogonal basis in the space of square integrable $G$-invariant kernels
$k(x_1, x_2)=k(gx_1, gx_2),\, x_1, x_2 \in Q,\, g\in G$.
They 
are eigenfunctions of the radial part of the Laplace--Beltrami operator 
on $Q$ and can be given explicitly, see, for example, \cite[p.~178]{19}, 
\cite[Chapters 2 and 17]{35}: 
\begin{equation}
\phi_l(Q,x_1,x_2) = \frac{P^{(\frac{d}{2} -1,\frac{d_0}{2} -1)}_l 
(\cos\theta(x_1,x_2))}
{P^{(\frac{d}{2} -1,\frac{d_0}{2} -1)}_l (1)}, \quad l\ge 0,\,\, x_1, x_2 \in Q,
\label{eq4.1}
\end{equation}
where $P^{(\alpha,\beta)}_n (t),\, t\in [-1,1],$ are  Jacobi polynomials 
of degree $n$.
They can be given by Rodrigues' formula 
\begin{equation}
P^{(\alpha,\beta)}_n(t) =
\frac{(-1)^n}{2^n n!} (1-t)^{-\alpha} (1+t)^{-\beta}
\frac{\dd^n}{\dd t^n} 
\left\{ (1-t)^{n+\alpha} (1+t)^{n+\beta} \right\}.
\label{eq9.11}
\end{equation}
Notice that 
$\vert P^{(\alpha,\beta)}_n(t)\vert \le P^{(\alpha,\beta)}_n(1)$ 
for $t\in [-1,1]$ and $\alpha \ge -1,\,  \beta \ge -1$, where
\begin{equation}
P^{(\alpha,\beta)}_n(1) =
\binom{\alpha +n}{n} =\frac{\Gamma (\alpha +n+1)}{\Gamma (l+1) \Gamma(\alpha 
+1)}\, . 
\label{eq8.23}
\end{equation}
Recall also that Jacobi polynomials $P^{(\alpha,\beta)}_l(t)$ form an orthogonal basis in the space of square integrable functions on the interval $[-1, 1]$ with
respect to the measure $(1-t)^{\alpha} (1+t)^{\beta}\dd t$. 
A detailed consideration of Jacobi polynomials can be found in \cite{1*, 34}.

The following expansion for the chordal metric \eqref{eq2.4} 
in terms of the zonal spherical functions \eqref{eq4.1}
was given in \cite[Lemma~4.1]{0}:
\begin{equation}
\tau(y_1,y_2)=\frac12\sum\nolimits_{l\ge 1}\, M_l\,C_l\,\left[\, 1-\phi_l(Q,x_1,x_2)\,\right],
\label{eq9.16**}
\end{equation}
where 
\begin{equation}
C_l = B((d+1)/2, l+d_0/2)\, 
\Gamma (l+1)^{-1}\, (1/2)_{l-1}\, P^{((\frac{d}{2} -1,\frac{d_0}{2} -1))}_l (1) \, .
\label{eq0.16**}
\end{equation}
and
\begin{equation*}
M_l=(2l - 1+(d+d_0)/2 ) 
\frac{\Gamma (l+1)\Gamma (l-1+(d+d_0)/2)}{\Gamma(l+d/2)
\Gamma (l+d_0/2)}\,  ,
\; \; l\ge 1,
\label{eq8.26}
\end{equation*}
where $\Gamma(\cdot)$ is the gamma function.
Here we use the notation
\begin{equation}
(a)_0 = 1,\quad (a)_k = a(a+1)\dots (a+k-1)=\frac{\Gamma(a +k)}{\Gamma(a)}
\label{eq8.26*}
\end{equation}
for the rising factorial powers.

In fact, the formula \eqref{eq9.16**} is merely the expansion of 
$\sqrt {(1-t)/2},\,t\in [-1, 1],$ by Jacobi polynomials, see 
the second equality in \eqref{eq2.4}.

The corresponding expansion for the symmetric difference metric \eqref{eq1.13} 
was established in \cite[Theorem~4.1(ii)]{31}:
\begin{equation}
\theta^{\Delta}(\xi, y_1,y_2) =
B(d/2,d_0/2)^{-1}\, \sum\nolimits_{l\ge 1}\, l^{-2} M_lA_l(\xi) 
\left[\, 1-\phi_l(Q,x_1,x_2)\,\right],
\label{eq9.16}
\end{equation}
where 
\begin{equation}
A_l(\xi)=\int^{\pi}_{0}
\left\{ P^{(\frac{d}{2},\frac{d_0}{2})}_{l-1} (\cos r)\right\}^2  
(\sin\frac 12 r)^{2d}(\cos \frac12 r)^{2d_0} 
\, \dd \xi (r). 
\label{eq9.16*}
\end{equation}

The proof is based on the observation 
that the term $\mu (\BBB(y_1,r)\cap \BBB(y_2,r))$ in the formula \eqref{eq1.14*} can be thought of as a convolution of the characteristic functions of the balls on the homogeneous space $Q(d,d_0)$. 
Since the space $Q(d,d_0)$ is of rank one, such a convolution can be calculated explicitly.

The series \eqref{eq9.16**} and \eqref{eq9.16} converge absolutely and uniformly.

\begin{proposition}\label{prop3.1}
The equality \eqref{eq2.8} is equivalent to
the following formulas
\begin{align}
\int^1_{-1}
\left(P^{(d/2, d_0/2)}_{l-1}(t)\right)^2& \, \left(1-t\right)^{d}
\left(1+t\right)^{d_0}\, \dd t
\notag
\\ 
= &\,\,\frac{2^{d+d_0+1}\, (1/2)_{l-1}}{((l-1)!)^2}\,B(d+1, d_0 +1)\, 
T_{l-1}(d/2, d_0/2)\,
 \label{eq0.81*} 
\end{align}
for all $l\ge 1$, where
\begin{align}
T_{l-1}(d/2, d_0/2)=
=\frac{\Gamma(d/2+l)\,\Gamma(d_0/2+l)\,\Gamma(d/2+d_0/2+3/2))}
{\Gamma(d/2+1)\,\Gamma(d_0/2+1)\,\Gamma(d/2+d_0/2+1/2+l)} \, .
\label{eq0.81}
\end{align} 
\end{proposition}
\begin{proof}
Since zonal spherical functions are mutually orthogonal, we conclude from
the expansions \eqref{eq9.16**} and \eqref{eq9.16} that the equality 
\eqref{eq2.8}
is equivalent to the formulas
\begin{align}
\gamma(Q)\,l^{-2}\, B(d/2, d_0/2)^{-1}\, A_l(\xi^{\natural})= C_l/2\, ,
\quad l\ge 1\, .
\label{eq0.03}
\end{align}

We wish to simplify these formulas to obtain \eqref{eq0.81}.
The integral \eqref{eq9.16*} with the special measure 
$\dd\xi^{\natural} (r)=\sin r\,\dd r$ takes the form
\begin{align}
A_l(\xi^{\natural})&=\int^{\pi}_{0}
\left\{ P^{(\frac{d}{2},\frac{d_0}{2})}_{l-1} (\cos r)\right\}^2  
(\sin\frac12 r)^{2d}(\cos \frac12 r)^{2d_0} 
\,\sin r\, \dd r
\notag
\\
&=2^{-d-d_0}\int^1_{-1}
\left(P^{(d/2, d_0/2)}_{l-1}(t)\right)^2 \,\left(1-t\right)^{d}
\left(1+t\right)^{d_0}\, \dd t \, .
\label{eq0.05}
\end{align}
Hence, the formulas \eqref{eq0.03} can be written as follows
\begin{align}
\int^1_{-1}
\left(P^{(d/2, d_0/2)}_{l-1}(t)\right)^2 \,
\left(1-t\right)^{d}&
\left(1+t\right)^{d_0}\, \dd t 
\notag
\\
=\, &\frac{2^{d+d_0+1}\, (1/2)_{l-1}}{((l-1)!)^2}\,B(d+1, d_0 +1)\, T^*,
\label{eq0.03*}
\end{align}
where
\begin{equation}
T^*=\frac{(l!)^2\, B(d/2, d_0/2)\,\, C_l}
{4\,(1/2)_{l-1}\,B(d+1, d_0 +1)\,\gamma (Q)}\, .
\label{eq0.06*}
\end{equation}

On the other hand, using \eqref{eq8.23} and \eqref{eq0.16**}, we find 
\begin{equation}
C_l=(l!)^{-1}\, (1/2)_{l-1}\frac{\Gamma (d/2+1/2)\,\Gamma (l+d/2)\,\Gamma 
	(l+d_0/2)}
{\Gamma (l+1/2+d/2+d_0/2)\,\Gamma (d/2)}\, .
\label{eq0.07}
\end{equation}
Substituting \eqref{eq0.07} and \eqref{eq1.33*} into \eqref{eq0.06*}, we obtain
\begin{align}
T^*=
& \pi^{-1/2}\,(d+d_0)^{-1}\,\frac{\Gamma (d+d_0 +2)}
{\Gamma (d+1)\,\Gamma (d_0 +1)}\,\times
\notag
\\
\notag
\\
&\times
\frac{\Gamma (d/2+1/2)\,\Gamma (l+d/2)\,\Gamma (d_0/2+1/2)\,\Gamma (l+d_0/2)}
{\,\Gamma (d/2 +d_0/2)\,\Gamma (l+d/2 +d_0/2 +1/2)}\, .
\label{eq0.08}
\end{align}
Applying the duplication formula for gamma function, see \cite[Sec.~1.5]{1*},
\begin{equation}
\Gamma (2z)=\pi^{-1/2}\,2^{2z-1}\, \Gamma (z)\,\Gamma(z+1/2) 
\label{eq0.09}
\end{equation}
to the first co-factor in \eqref{eq0.08}, we find  
\begin{align}
\pi^{-1/2}\,(d+d_0)^{-1}\,&\frac{\Gamma (d+d_0 +2)}
{\Gamma (d+1)\,\Gamma (d_0 +1)}
\notag
\\
\notag
\\
\,
=\,\,&\frac{\Gamma (d/2 +d_0/2)\,\Gamma (d/2 +d_0/2 +3/2)}
{\Gamma (d/2 +1/2)\,\Gamma (d_0/2 +1)\,\Gamma (d_0/2 +1/2)\,\Gamma (d_0/2 
	+1)}\, ,
\label{eq0.010}
\end{align}
where the relation $\Gamma (z+1) = z\Gamma (z)$ with $z=d/2+d_0/2$
has been used also.

Substituting \eqref{eq0.010} into \eqref{eq0.08}, we find that 
$T^*=T_{l-1}(d/2, d_0/2)$.	
\end{proof}


By Proposition~4.1, the proof of Theorem~3.1 is reduced to the calculation of
integral \eqref{eq0.81}. We shall consider the following more general integral.
\begin{lemma}\label{lem3.1}
For all $n\ge 0$, $\Re\alpha >-1/2$ and $\,\Im\beta >-1/2$, we have
\begin{align}
\int^1_{-1}\left(P^{(\alpha, \beta)}_{n}(t)\right)^2& 
(1-t)^{2\alpha}
(1+t)^{2\beta}\, \dd t 
\notag
\\
=&\,\, \frac{2^{2\alpha+2\beta+1}\, (1/2)_{n}}{(n!)^2}\,B(2\alpha+1, 2\beta 
+1)\,
T_n(\alpha, \beta),
\label{eq0.01*}
\end{align}
where
\begin{align}
T_n(\alpha, \beta)=&
\,\,\frac{(\alpha +1)_n\, (\beta +1)_n}{(\alpha +\beta +3/2)_n}
\notag
\\
=&\,\,\frac{\Gamma(\alpha+n+1)\,\Gamma(\beta+n+1)\,
\Gamma(\alpha+\beta+3/2))}
{\Gamma(\alpha+1)\,\Gamma(\beta+1)\,\Gamma(\alpha +\beta +3/2+n)}  
\label{eq0.01}
\end{align} 
is a rational function of $\alpha$ and $\beta$.
\end{lemma}

The integral \eqref{eq0.01*} converges for $\Re\alpha >-1/2$ and $\,\Re\beta 
>-1/2$,
and represents in this region 
a holomorphic function of two complex variables.  
The equality \eqref{eq0.01*} defines an analytic continuation of 
the integral \eqref{eq0.01*} to $\alpha\in\Cc$ and $\beta\in\Cc$.

The proof of Lemma~4.1 is given in Section~5. 

Now we are in position to prove Theorem~2.1. 
For $n=l-1,\,\alpha = d/2$ and $\beta = d_0/2$, the formulas \eqref{eq0.01*}
coincide with the formulas \eqref{eq0.81*}. Therefore the equality \eqref{eq2.8}
holds by Proposition~4.1. 
This completes the proof of Theorem~3.1.

The possibility to reduce the proof of  $L_2$-invariance 
principles 
to the calculation of integral \eqref{eq0.81*} has been briefly mentioned 
in \cite[Remark~4.2]{0}. In the present paper such an approach is realized.


\section{Proof of Lemma~4.1}\label{sec5}

Lemma~4.1 follows from Lemma~5.1 and Lemma~5.2 given below. 

We use the notation
\begin{equation}
\langle a \rangle_0 = 1,\quad \langle a \rangle_k = a(a-1)\dots (a-k+1)=(-1)^k\,(-a)_k\, . 
\label{eq8.26**}
\end{equation}
for the falling factorial powers. 

\begin{lemma}\label{lem4.1}
For all $n\ge 0$, $\Re\alpha >-1/2$ and $\,\Re\beta >-1/2$, we have
\begin{align}
\int^1_{-1}\left(P^{(\alpha, \beta)}_{n}(t)\right)^2& 
(1-t)^{2\alpha}
(1+t)^{2\beta}\, \dd t 
\notag
\\
& 
=\,\frac{2^{2\alpha+2\beta+1}}
{(n!)^2}\,B(2\alpha+1, 2\beta +1)\,
\frac{W_n(\alpha, \beta)}{(2\alpha +2\beta +2)_{2n}} \, ,
\label{eq0.011}
\end{align}
where
\begin{align}
W_n(\alpha, \beta&)
\notag
\\
=&\sum\nolimits_{k=0}^{2n}\frac{(-1)^{n+k}}{k!} 
\langle 2n \rangle_k \,\langle \alpha +n \rangle_k \,\langle \beta +n \rangle_{2n-k}\,
(2\alpha +1)_{2n-k} \, (2\beta +1)_k  
\label{eq0.012}
\end{align} 
is a polynomial of $\alpha$ and $\beta$.
\end{lemma}
\begin{proof}
Using Rodrigues' formula \eqref{eq9.11}, we can write
\begin{align}
\int^1_{-1}\left(P^{(\alpha, \beta)}_{n}(t)\right)^2 
(1-t)^{2\alpha}
(1+t)^{2\beta}\, \dd t =  \Big (\, \frac{1}{2^n\,n!}\,\Big)^2 \, I_n(\alpha, \beta)\,.
\label{eq0.011a}
\end{align}
where
\begin{align}
I_n(\alpha, \beta)=\int^1_{-1}\Big(\frac{\dd^n}{\dd t^n} 
\left [ (1-t)^{n+\alpha} (1+t)^{n+\beta} \right ] \Big)^2
\, \dd t \,.
\label{eq0.012a}
\end{align}
Integrating in \eqref{eq0.012a} $n$ times by part, we obtain
\begin{align}
I_n(&\alpha, \beta)
\notag
\\
&=(-1)^n\,\int^1_{-1}\left( (1-t)^{n+\alpha} (1+t)^{n+\beta} \right)\,
\frac{\dd^{2n}}{\dd t^{2n}} \left( (1-t)^{n+\alpha} (1+t)^{n+\beta} \right)  
\, \dd t \,,
\label{eq0.013a}
\end{align}
since all terms outside the integral vanish.
By Leibniz's rule,
\begin{align*}
\frac{\dd^{2n}}{\dd t^{2n}}& \left(  (1-t)^{n+\alpha}\, (1+t)^{n+\beta} \right) 
\notag
\\ 
&=\sum\nolimits_{k=0}^{2n}\, {2n\choose k}\,\,
\frac{\dd^k}{\dd t^k} (1-t)^{n+\alpha}\,\,
\frac{\dd^{2n-k}}{\dd t^{2n-k}}\, (1+t)^{n+\beta} \, ,
\end{align*}
where ${2n\choose k}=\langle 2n \rangle_k /k!\,$ and 
\begin{align*}
\frac{\dd^k}{\dd t^k} (1-t)^{n+\alpha}=(-1)^k\,\langle \alpha +n \rangle_k\, (1-t)^{n-k+\alpha} \, ,
\\        
\frac{\dd^{2n-k}}{\dd t^{2n-k}} (1+t)^{n+\beta}=\langle \beta +n \rangle_{2n-k}\, (1+t)^{-n+k+\beta}\, .
\end{align*}
Substituting these formulas into \eqref{eq0.013a}, we obtain
\begin{align}
I_n(\alpha, \beta&)
\notag
\\
=&\,2^{2\alpha +2\beta +2n+1}\sum\nolimits_{k=0}^{2n}\frac{(-1)^{n+k}}{k!} 
\langle 2n \rangle_k \,\langle \alpha +n \rangle_k \,\langle \beta +n \rangle_{2n-k}\,
\,I_n^{(k)}(\alpha ,\beta)\, ,  
\label{eq0.012aa}
\end{align} 
where
\begin{equation} 
I_n^{(k)}(\alpha ,\beta)=B(2\alpha +2n-k+1, 2\beta +k+1).
\label{eq0.0ab} 
\end{equation}
Here we have used the following Euler's integral 
\begin{align}
2^{1-a-b}\int_{-1}^{1} (1-t)^{a-1}\,(1+t)^{b-1}\,\dd t =
B(a,b)= \frac{\Gamma (a)\Gamma (b)}{\Gamma (a+b)}
\label{eq0.02}
\end{align}
with $\Re a>0,\, \Re b>0$. 

The formula \eqref{eq0.0ab} can be written as follows
\begin{align}
&I_n^{(k)}(\alpha, \beta)=
\frac{\Gamma (2\alpha +2n-k+1)\,\Gamma (2\beta +k+1)}
{\Gamma (2\alpha +2\beta +2n+2)}
\notag
\\
\notag
\\
=&\frac{\Gamma (2\alpha +2n-k+1)}{\Gamma (2\alpha +1)}
\frac{\Gamma (2\beta +k+1)}{\Gamma (2\beta +1)}
\frac{\Gamma (2\alpha +1)\,\Gamma (2\beta +1)}{\Gamma (2\alpha +2\beta +2)}
\frac{\Gamma (2\alpha +2\beta +2)}{\Gamma (2\alpha +2\beta +2n+2)}
\notag
\\
\notag
\\
=&\frac{(2\alpha +1)_{2n-k}\, (2\beta +1)_k}{(2\alpha +2\beta +2)_{2n}}\, 
B(2\alpha +1, 2\beta +1 )\, .
\label{eq0.014}
\end{align} 

Combining the formulas \eqref{eq0.014}, \eqref{eq0.012aa} and \eqref{eq0.011a},
we obtain \eqref{eq0.011}.
\end{proof}

The next Lemma~4.2 is more specific, it relies on 
Watson's theorem for generalized hypergeometric series, see \cite{1*, 33a}.
We consider the series of the form 
\begin{align}
_3F_2(a, b, c; d, e; z)=
\sum\nolimits_{k\ge 0}\frac{(a)_k\, (b)_k\, (c)_k\, z}{(d)_k\, (e)_k\, k!}\, ,  
\label{eq0.015}
\end{align}
where neither $d$ nor $e$ are equal to negative integers.  
The series absolutely converges for $\vert z\vert\le 1$, if 
$\Re (d+e)>\Re (a+b+c)$. 
The series \eqref{eq0.015} terminates, if one of the numbers  $a, b, c$ 
is a negative integer. 

\bfseries Watson's theorem.\mdseries {\it We have}
\begin{align}
_3F_2(a, &b, c; (a+b+1)/2, 2c; 1)
\notag
\\
=&\frac{\Gamma (1/2)\,\Gamma (c+1/2)\,\Gamma ((a+b+1)/2)\,\Gamma 
(c-(a+b-1)/2)\,}
{\Gamma ((a+1)/2)\,\Gamma ((b+1)/2)\,\Gamma (c-(a-1)/2)\,\Gamma (c-(b-1)/2)\,} 
\, .  
\label{eq0.016}
\end{align}  
{\it provided that}  
\begin{equation}
\Re\, (2c - a -b+1) > 0. 
\label{eq0.017}
\end{equation}

The condition \eqref{eq0.017} ensures the convergence of hypergeometric series 
in \eqref{eq0.016}. Furthermore, this condition is necessary for the truth 
of equality \eqref{eq0.016} even in the case of terminated series. The proof of
Watson's theorem can be found in \cite[Therem~3.5.5]{1*}, 
\cite[p.54, Eq.(2.3.3.13)]{33a}.

\begin{lemma}\label{lem4.2}
For all $n\ge 0$, $\alpha \in \Cc$ and $\beta\in\Cc$,	
the polynomial \eqref{eq0.012} is equal to
\begin{align}
W_n(\alpha, \beta)=&2^{2n}\,(\alpha +1)_n\,(\beta +1)_n\,(\alpha +\beta +1)_n
\notag
\\
=&2^{2n}\,\frac{\Gamma(\alpha +1+n)\,\Gamma(\beta +1+n)\,\Gamma(\alpha +\beta 
+1+n)}
{\Gamma(\alpha +1)\,\Gamma(\beta +1)\,\Gamma(\alpha +\beta +1)}\, .
\label{eq0.018}
\end{align} 
In particular, 
\begin{align}
\frac{W_n(\alpha, \beta)}{(2\alpha +2\beta +2)_{2n}}=
\frac{(\alpha +1)_n\, (\beta +1)_n}{(\alpha +\beta +3/2)_n}\, .
\label{eq0.018a}
\end{align} 
\end{lemma}
\begin{proof}
Since $W_n(\alpha, \beta)$ is a polynomial, it suffers to check 
the equality \eqref{eq0.018} for $\alpha$ and $\beta$ in an open subset in $\Cc 
^2$. As such a subset we 
shall take 
the following region
\begin{equation} 
\O=\{\,\alpha,\,\beta\, : \Re\alpha <0,\,\,\Re\beta <0, \,\, \Im\alpha 
>0,\,\,\Im\beta >0\,\}.
\label{0.019}   
\end{equation} 
For $\alpha$ and $\beta$ in $\O$, the co-factors in terms in \eqref{eq0.012}
may be rearranged as follows: 
\begin{equation}
\left.
\begin{aligned}
&\langle 2n \rangle_k =(-1)^k\,(-2n)_k \, ,\quad
\langle \alpha +n \rangle_k = (-1)^k\, (-\alpha -n)\, , \\
&\langle \beta +n \rangle_{2n-k}=(-1)^k\,(-\beta -n)_{2n-k}=
\frac{(-\beta -n)_{2n}}{(\beta +1-n)_k}\, ,\quad \\
&(2\alpha +1)_{2n-k} = \frac{(-1)^k(2\alpha +1)_{2n}}{(-2\alpha -2n)_k}\, ,
\end{aligned}
\label{eq1.34a}
\right\}
\end{equation}
Here we have used the following elementary relation for the rising factorial 
powers
\begin{equation}
(a)_{m-k}=\frac{(-1)^k\,(a)_m}{(1-a-m)_k}\, ,\quad m\ge 0\, ,\,\, 0\le k \le m 
\, .
\end{equation}
Substituting \eqref{eq1.34a} into \eqref{eq0.012}, we find that 
\begin{align}
W_n(\alpha, \beta)&=
(-1)^n\,(2\alpha +1)_{2n}\,(-\beta -n)_{2n}\,\FFF_n(\alpha , \beta )
\notag
\\
&=\frac{(-1)^n\,\Gamma (2\alpha +1+2n)\,\Gamma (-\beta +n)}
{\Gamma (2\alpha +1)\,\Gamma (-\beta -n)}\,\, \FFF_n(\alpha , \beta )\, ,
\label{eq0.35a}
\end{align} 
where 
\begin{equation} 
\FFF_n(\alpha , \beta )=\sum\nolimits_{k=0}^{2n}\frac
{(-2n)_k\,(2\beta +1)_k\,(-\alpha -n)_k}
{(\beta +1-n)_k\,(-2\alpha -2n)_k\,k!}
\label{eq0.355}
\end{equation}
In view of the definition \eqref{eq0.015}, we have
\begin{align}
\FFF_n(\alpha , \beta )=\,
_3F _2\,(-2n, 2\beta +1, -\alpha -1; \beta 
+1-n, -2\alpha -2n; 1)\, .
\label{eq0.40}
\end{align} 
The parameters in hypergeometric series  
\eqref{eq0.40} are identical  with those in \eqref{eq0.016} for
$a=-2n, \, b=2\beta +1, \, c= -\alpha -n$,
and in this case, $(a+b+1)/2 = 2\beta +1 +n$, $2c = -2\alpha -2n$.
The condition \eqref{eq0.017} also holds for $\alpha$ and $\beta$ 
in the region $\O$, since $\Re\, (2c - a -b+1)=\Re\, (-2\alpha -2\beta ) >0$.
Therefore, Watson's theorem \eqref{eq0.016} can be applied to obtain 
\begin{align}
\FFF_n(\alpha , \beta )=\,
\frac{\Gamma(1/2)\,\Gamma(-\alpha -n-1/2)\,\Gamma(\beta +1-n)\,
\Gamma(-\alpha -\beta)}
{\Gamma(-n+ 1/2)\,\Gamma(\beta +1)\,\Gamma(-\alpha +1/2)\,
\Gamma(-\alpha -\beta -n)}
\, .
\label{eq0.41}
\end{align} 

Substituting the expression \eqref{eq0.41} into \eqref{eq0.35a} 
, we may write
\begin{equation}
W_n(\alpha, \beta )=\, c_0\,\, c_1(\alpha )\,\, c_2(\beta )\,\,c_3(\alpha 
+\beta )\, ,
\label{eq0.42}
\end{equation}
where
\begin{equation}
\left.
\begin{aligned}
&c_0 = \frac{(-1)^n\, \Gamma (1/2)}{\Gamma (-n+1/2)}\, ,\\ 
&c_1(\alpha)=\frac{\Gamma (2\alpha +2n+1)\,\Gamma (-\alpha -n+1/2)}
{\Gamma (2\alpha +1)\,\Gamma (-\alpha +1/2)}\,,\quad \\
&c_2(\beta)= \frac{\Gamma (\beta +1-n)\,\Gamma (-\beta +n)}
{\Gamma (\beta +1)\,\Gamma (-\beta -n)}\, ,\\
&c_3(\alpha +\beta)=\frac{\Gamma (-\alpha -\beta )}
{\Gamma (-\alpha -\beta -n)}\, .
\end{aligned}
\label{eq1.45}
\right\}
\end{equation}

Using the duplication formula \eqref{eq0.09} and reflection formulas,
see \cite[Sec.~1.2]{1*}, 
\begin{equation}
\Gamma (1-z)\Gamma (z)\, =\, \frac{\pi}{\sin \pi z}\, , \qquad 
\Gamma (1/2 -z)\Gamma (1/2 +z)\, =\, \frac{\pi}{\cos \pi z} \, ,
\label{eq0.43}    
\end{equation}
we may rearrange the expressions in \eqref{eq1.45} as follows.
For $c_0$, we have
\begin{align*}
&c_0 = \frac{(-1)^n\,\Gamma (1/2)^2}{\Gamma (-n+1/2)\,\Gamma (n+1/2)}\, 
\frac{\Gamma (n+1/2)}{\Gamma (1/2)}=(1/2)_n \, ,
\end{align*} 
since $\Gamma(1/2)=\sqrt\pi $.
For $c_1(\alpha)$ and $c_2(\beta)$, we have
\begin{align*}
c_1(\alpha)=&2^{2n}\,
\frac{\Gamma(\alpha +n+1)\,\Gamma(\alpha +n+1/2)\,\Gamma(-\alpha -n+1/2)}
{\Gamma(\alpha +1)\,\Gamma(\alpha +1/2)\,\Gamma(-\alpha +1/2)}
\\
\\
=&2^{2n}\,
\frac{\cos\pi\alpha\,\Gamma(\alpha +n+1)}{\cos\pi(\alpha+n)\,\Gamma(\alpha +1)}
=2^{2n}\,(-1)^n\,(\alpha +1)_n
\end{align*}
and
\begin{align*}
c_2(\beta) = \frac{\Gamma (\beta +1-n)\,\Gamma (-\beta +n)}
{\Gamma (\beta +1)\,\Gamma (-\beta -n)}
=\frac{\sin\pi (\beta +n)\,\Gamma (\beta +1+n)}
{\sin\pi (\beta -n)\,\Gamma (\beta +1)}=(\beta +1)_n \, .
\end{align*}
Finally,
\begin{align*}
c_3(\alpha +\beta)=\frac{\sin\pi (\alpha +\beta)\,\Gamma (\alpha +\beta 
+1+n)}
{\sin\pi (\alpha +\beta +n)\,\Gamma (\alpha +\beta +1)}=(-1)^n\, 
(\alpha +\beta +1)_n
\, .
\end{align*}

Substituting these expressions into \eqref{eq0.42}, we obtain \eqref{eq0.018}.

It follows from \eqref{eq8.26*} and the duplication formula \eqref{eq0.09}
that 
\begin{equation}
(2\alpha +2\beta +2)_{2n}=
2^{2n}\,(\alpha +\beta +1)_n\, (\alpha +\beta +3/2)_n \, .
\label{eq0.47}
\end{equation}
Using \eqref{eq0.018} together with \eqref{eq0.47}, we obtain \eqref{eq0.018a}.
\end{proof}

Now it suffers to substitute \eqref{eq0.018a} into \eqref{eq0.011} 
to obtain  the formula \eqref{eq0.01*}. The proof of Lemma~4.1 
is complete.


\end{document}